\documentclass{amsart}         
\usepackage{amsmath,amsfonts,amsthm,amssymb}    
\usepackage[pageanchor,colorlinks]{hyperref}               
\usepackage[mathcal]{euscript}                  
\usepackage{enumitem}                           

\usepackage{datetime}                            

\newtheorem{theorem}{Theorem}[section]

\newtheorem{lemma}[theorem]{Lemma}
\newtheorem{proposition}[theorem]{Proposition}

\theoremstyle{definition}
\newtheorem{definition}[theorem]{Definition}
\theoremstyle{remark}
\newtheorem{remark}[theorem]{Remark}
 
\numberwithin{equation}{section}

\def\R{{\mathbb R}}
\def\N{{\mathbb N}}

\def\E{{\mathbb E}}
\def\P{{\mathbb P}}

\def\ra{\rightarrow}
\def\eqdist{\stackrel{(d)}{=}}
\renewcommand{\|}{\Vert}

\def\snp{\mathbb{S}_p^{n-1}}

\def\alp{X}
\def\anp{\alp^{(n,p)}}

\def\ynp{Y^{(n,p)}}
\def\lnp{\mathcal{L}_{n,p}}
\def\lnpy{L_{n,p}^Y}

\begin{document}


\title{A conditional limit theorem for high-dimensional $\ell^{p}$ spheres}

\author{Steven Soojin Kim}
\address{Division of Applied Mathematics \\ Brown University \\ 182 George St., Providence, RI }
\email{steven\_kim@brown.edu}
\thanks{SSK was partially supported by a Department of Defense NDSEG fellowship.}

\author{Kavita Ramanan}
\address{Division of Applied Mathematics\\ Brown University  \\ 182 George St., Providence, RI }
\email{kavita\_ramanan@brown.edu}
\thanks{KR was partially supported by ARO grant W911NF-12-1-0222 and NSF grant DMS 1407504.}

\thanks{SSK and KR would also like to thank Microsoft Research New England for their hospitality during the Fall of 2014, when some of this work was carried out.}

\begin{abstract} 
The study of high-dimensional distributions is of interest in probability theory, statistics and asymptotic convex geometry, where the object of interest is the uniform distribution on a convex set in high dimensions. The $\ell^p$ spaces and norms are of particular interest in this setting. In this  paper, we establish a limit theorem for distributions on $\ell^p$ spheres, conditioned on a rare event, in a high-dimensional geometric setting. As part of our proof, we establish a certain large deviation principle that is also relevant to the study of the tail behavior of random projections of $\ell^p$ balls in a high-dimensional Euclidean space. 
\end{abstract}


\maketitle 

\vspace{-0.3in}

\section{Introduction} \label{sec-intro}

The classical Poincar\'e-Maxwell-Borel lemma (see, e.g., \cite{diaconis1987dozen}) states that for any fixed number of coordinates of a random vector drawn uniformly from a high-dimensional Euclidean sphere, the joint distribution of the coordinates is approximately standard Gaussian. An analogous description (recalled precisely in Proposition \ref{lem-pbm} to follow) holds in the setting of $\ell^p$ spheres, for general $p\in[1,\infty]$, which yields an asymptotic probabilistic characterization of the geometry of an $\ell^p$ sphere. 

In this paper, we establish a \emph{conditional} analog of the aforementioned \emph{unconditional} result. Roughly stated, we prove the following asymptotic probabilistic description of the geometry of an $\ell^p$ sphere under an $\ell^q$ norm constraint: 
\begin{quote}
\itshape For $1 \le q<p < \infty$, in high dimensions, a random point on the $\ell^p$ sphere of $\R^n$ conditioned on having a sufficiently small $\ell^q$ norm is close  (in the sense of distribution) to a random point drawn from an appropriately scaled $\ell^q$ sphere of $\R^n$.
\end{quote}
That is, conditioning on a small $\ell^q$ norm induces a \emph{probabilistic} change that admits a \emph{geometric} interpretation.

As part of our proof, we establish a large deviation principle for the empirical measure of the coordinates of a random vector distributed according to the uniform (cone or surface) measure on a suitably scaled $\ell^p$ sphere in $\R^n$, for $p\in[1,\infty]$, with respect to the $q$-Wasserstein topology for every $q<  p$. This large deviation principle is also of relevance to the study of the tail behavior of random projections of $\ell^p$ balls.

The remainder of the paper is organized as follows: 
\begin{itemize}
\item \textsc{Sect. \ref{sec-main}:} We precisely state our main result, Theorem \ref{cor-asymp}, a conditional limit theorem for $\ell^p$ spheres.

\item \textsc{Sect. \ref{sec-ldpres}:} We establish a large deviation principle (LDP) for the sequence of empirical measures of the coordinates of the random vector $n^{1/p}X^{(n,p)}$ distributed uniformly on the sphere $n^{1/p}\, \snp$. We emphasize that the aforementioned LDP holds in the $q$-Wasserstein topology for every $q < p$ (which is stronger than the usual weak topology), and as summarized in Remark \ref{rmk-wasswhy}, we crucially utilize this fact in order to prove Proposition  \ref{prop-clln}, as well as to establish a variational formula in a related paper \cite[Theorem 2.7]{gkr2}.

\item \textsc{Sect.\ \ref{sec-gibbscond}:} We apply the LDPs of Sect.\ \ref{sec-ldpres} and recall the ``Gibbs conditioning principle" in order to establish the proof of our main result Theorem \ref{cor-asymp}.
\end{itemize}

\section{Statement of Main Result}\label{sec-main}

 To precisely state the unconditional limit theorem for $\ell^p$ spheres, for $n\in \N$, let $\|\cdot\|_{n,p}$ denote the $\ell^p$ norm on $\R^n$, and write $\snp \doteq \{a\in \R^n : \|a\|_{n,p} = 1\}$ for the unit $\ell^p$ sphere in $\R^n$. There are two natural notions of ``uniform" measure on the sphere $\snp$: the \emph{cone measure} $\gamma_{n,p}$, and the \emph{surface measure} $\sigma_{n,p}$, defined  in Definitions \ref{def-cone} and \ref{def-surf}, respectively. Assume that all random variables we define are supported on a common probability space $(\Omega,\mathcal{F},\P)$. 

\begin{definition}
For $n\in \N$, $p\in [1,\infty]$, we write $\anp = (\anp_1, \dots, \anp_n)$ for an $n$-dimensional random vector in $\snp$, distributed according to \emph{either} the surface measure $\sigma_{n,p}$ \emph{or} the cone measure $\gamma_{n,p}$ (i.e., all of our results hold under either distribution for $\anp$, unless explicitly specified otherwise).
\end{definition}

Let $\mu_p$ denote the \emph{generalized Gaussian} distribution with location 0, shape $p$, and scale $p^{1/p}$. That is, for $p\in[1,\infty)$,
\begin{equation}\label{gengsn}
  \mu_p(dy) \doteq \frac{1}{2p^{1/p}\Gamma(1+\tfrac{1}{p})} e^{-|y|^p/p}\,dy, \quad y\in \R.
\end{equation}
For $p =\infty$, define
\begin{equation*}
  \mu_\infty(dy) = \tfrac{1}{2}\mathbf{1}_{[-1,1]}(y)\, dy, \quad y\in \R.
\end{equation*}
The aforementioned ``$\ell^p$ version" of the Poincar\'e-Maxwell-Borel lemma is stated as follows. Let $\Rightarrow$ denote weak convergence of probability measures.

\begin{proposition}\label{lem-pbm}
Fix $p \in [1,\infty]$ and $k\in \N$. As $n\ra\infty$,
\begin{equation*}
\textnormal{Law}\left[n^{1/p}(\anp_1,\dots,\anp_k)\right] \Rightarrow \mu_p^{\otimes k}.
\end{equation*}
\end{proposition}

This result is due to \cite[Theorems 4.1 and
4.4]{rachev1991approximate} in the case of the cone measure, and
\cite{mogul1991finetti} in the case of the surface measure. In
addition, the work \cite[Theorems 3 and 4]{naor2003projecting} offers a simplification of the proof for both the cone and surface measures. 

Note that the scaling $n^{1/p}$ is natural from a geometric perspective, due to the following reasoning: the typical coordinate of the random vector $\anp$ (uniformly distributed on the unit $\ell^p$-sphere of dimension $n$) has length $n^{-1/p}$ (since the $n$-fold sum of $p$-th powers of the coordinates is constrained to sum to 1). Therefore, in order to have the joint law of the coordinates of $\anp$ be non-trivial, one must scale the random vector $\anp$ appropriately; that is, by a factor of $n^{1/p}$.

Our main result, Theorem \ref{cor-asymp}, yields a conditional analog of Proposition \ref{lem-pbm}. To state this precisely, we first set some notation. For $\beta > 0$ and $q\in[1,\infty]$, let  
\begin{equation}\label{uncondlaw}
 \lambda_{n,q}^{(\beta,k)} \doteq \textnormal{Law}\left[ \beta^{1/q}n^{1/q}(X_1^{(n,q)}, \dots, X_k^{(n,q)}) \right].
\end{equation}
In addition,  let $\widehat{\lambda}_{n,p|q}^{(\beta,k)}$ denote the conditional law,
\begin{equation}\label{condlaw}
  \widehat{\lambda}_{n,p|q}^{(\beta,k)} \doteq  \P\left(\left. n^{1/p}(X_1^{(n,p)},\dots,X_k^{(n,p)}) \in \,\cdot\, \, \right| \,\,\tfrac{1}{n} \|n^{1/p}X^{(n,p)}\|_{n,q}^q \le \beta \right).
\end{equation}
For $1\le q < p <\infty$,  define
\begin{equation}\label{betapq}
\beta_{p,q} \doteq \frac{1}{q}\left(\frac{\Gamma(\tfrac{1}{q})}{\Gamma(\tfrac{p+1}{q})}\right)^{q/p}.
\end{equation}
Lastly, let $\rho$ be a metric which metrizes the topology of weak convergence of probability measures (e.g., L\'evy-Prohorov, bounded Lipschitz) \cite[Appendix III]{billingsley2013convergence}. 

\begin{theorem}[Conditional limit theorem]\label{cor-asymp}
Assume $1 \le q < p < \infty$, and $\beta \le \beta_{p,q}$.  Then, for fixed $k\in \N$,
\begin{equation}\label{rhoclose}
\lim_{\epsilon \downarrow 0} \lim_{n\ra\infty}  \rho \left(   \, \widehat{\lambda}_{n,p|q}^{(\beta+\epsilon,k)} \,\, , \,\,  \lambda_{n,q}^{(\beta,k)} \, \right) = 0.
\end{equation}
\end{theorem}

We defer the proof of Theorem \ref{cor-asymp} to Sect.\
\ref{sec-gibbscond}. It is worth emphasizing that a suitable choice 
of the conditioning event, as in the expression \eqref{condlaw}, is required
to obtain a meaningful conditional limit theorem.  Theorem \ref{cor-asymp} shows that for appropriate values of $\beta$, the conditioning event
  \begin{equation}\label{anb}
  A_{n,\beta}\doteq\{a\in \R^n: \tfrac{1}{n} \|a\|_{n,q}^q \le \beta \},
\end{equation}
provides a geometric example of a \emph{rare event} in high dimensions
for which such a conditional limit theorem can be established, in the
sense that  $\P(n^{1/p}\anp \in A_{n,\beta+\epsilon})$ decays to zero
as $n\ra\infty$, in  fact, at an exponential rate.  A basic intuition from the theory of {large deviations} suggests that for large $n$, given a rare event such as $A_{n,\beta+\epsilon}$, the conditional law will be close to a ``tilted" measure under which the event $A_{n,\beta+\epsilon}$ is in fact typical. The goal of finding a measure under which the event $A_{n,\beta+\epsilon}$ is typical motivates the scaling $\beta^{1/q}$ and the random variable $X^{(n,q)}$ in \eqref{uncondlaw}. That is, it follows from the definition  of $\sigma_{n,q}$ and $\gamma_{n,q}$ that for all $n\in \N$ and $\epsilon > 0$,   $\P(\beta^{1/q} n^{1/q}X^{(n,q)}\in A_{n,\beta+\epsilon}) =1$. The particular choice of $\beta \le \beta_{p,q}$ is clarified in Remarks \ref{rmk-beta}, \ref{rmk-beta2}, and \ref{rmk-beta3}.

The preceding paragraph's general ideas of conditioning and tilting
have been realized to great effect in statistical mechanics, where
large deviations theory can describe the most probable state of a
system of particles under an energy constraint (see, e.g., the surveys
\cite{bertini2014macroscopic, ellis2012entropy}). Of particular
utility is the so-called ``Gibbs conditioning principle", which
transforms an LDP for empirical measures into a statement about the
most probable behavior of the underlying sequence of random variables,
conditional on a rare event.   A central motivation of this work is to
employ such a conditional probabilistic perspective in a
high-dimensional geometric setting, by investigating how LDPs can
inform the analysis of ``geometric" rare events in high dimensions,
like $A_{n,\beta}$ of \eqref{anb}.   

More generally, this work continues a theme of work at the intersection of large deviations theory and high-dimensional geometry. We refer the interested reader to \cite{gkr2} for large deviation analysis of $k$-dimensional projections of random vectors drawn from $\ell^p$ balls, \cite{kabluchko2017high} for central limit theorem analysis of such projections, and \cite{alonso2016large} for large deviation analysis of $k_n$-dimensional projections of random vectors where the lower-dimension $k_n$ grows with the size of the higher dimension $n$.


\section{Large deviations results} \label{sec-ldpres}

In this section, we state and prove a large deviation principle for the following empirical measure: for $n\in\N$, $p\in[1,\infty]$, define
\begin{equation}\label{empir}
\lnp \doteq \frac{1}{n}\sum_{i=1}^n \delta_{n^{1/p}\anp_i}.
\end{equation}
In Sect.\ \ref{ssec-bkgrnd}, we review some basic elements of large deviation theory. In Sect.\ \ref{sec-empcone} (resp., Sect.\ \ref{sec-surface}), we prove the LDP for $(\lnp)_{n\in\N}$ in the Wasserstein topology, assuming $\anp$ is distributed according to the cone measure  (resp., the surface measure). This Sanov-type LDP  complements the existing Glivenko-Cantelli-type law of large numbers and Donsker-type  central limit theorem for $(\lnp)_{n\in\N}$ under the cone measure \cite[Theorem 1]{spruill2007asymptotic}. 

\subsection{Background on large deviations}\label{ssec-bkgrnd}

For a broad review of large deviations, we refer to \cite{DemZeibook}. In particular, recall the following definition:

\begin{definition}\label{def-ldp}
Let $\mathcal{X}$ be a topological space equipped with its Borel $\sigma$-algebra, and let $\mathcal{P}(\mathcal{X})$ denote the space of probability measures on $\mathcal{X}$. A sequence of probability measures $(\mu_n)_{n\in \N}\subset\mathcal{P}(\mathcal{X})$ is said to satisfy a \emph{large deviation principle} with a \emph{rate function} $\mathbb{I}:\mathcal{X}\ra[0,\infty]$ if $\mathbb{I}(\cdot)$ is lower semicontinuous, and for all Borel measurable sets $\Gamma\subset \mathcal{X}$, 
\begin{equation*}
  -\inf_{x\in \Gamma^\circ} \mathbb{I}(x) \le  \liminf_{n\ra\infty} \tfrac{1}{n} \log \mu_n(\Gamma^\circ) \le \limsup_{n\ra\infty} \tfrac{1}{n} \log \mu_n(\bar \Gamma) \le -\inf_{x\in \bar{\Gamma}} \mathbb{I}(x),
\end{equation*}
where $\Gamma^\circ$ and $\bar\Gamma$ denote the interior and closure of $\Gamma$, respectively. Furthermore, $\mathbb{I}$ is said to be a \emph{good rate function} if it has compact level sets. Similarly, we say the sequence of $\mathcal{X}$-valued random variables $(\xi_n)_{n\in \N}$ satisfies an LDP if the sequence of laws $(\P \circ \xi_n^{-1})_{n\in\N}$ satisfies an LDP.
\end{definition}

In the remaining sections, we will also make occasional reference to the following related notion.

\begin{definition}\label{def-etight}
A sequence of probability measures $(\mu_n)_{n\in\N} \subset \mathcal{P}(\mathcal{X})$ is said to be \emph{exponentially tight} if for every $\alpha < \infty$, there exists a compact set $K_\alpha \subset \mathcal{X}$ such that 
\begin{equation*}
  \limsup_{n \ra\infty} \tfrac{1}{n}\log\mu_n(K_\alpha^c) < -\alpha.
\end{equation*}
\end{definition}

There are two particularly well known examples of LDPs. For one, \emph{Cram\'er's theorem} (see, e.g., \cite[Sect. 2.2.2]{DemZeibook}) establishes an LDP for the sum of i.i.d.\ random variables, with the rate function shown to be the so-called Legendre transform of the log moment-generating function of the underlying distribution. For another, \emph{Sanov's theorem} (see, e.g., \cite[Sect 6.2]{DemZeibook}) establishes an LDP for the empirical measure of i.i.d.\ random variables, with the rate function shown to be the so-called relative entropy.

For $\nu,\mu\in\mathcal{P}(\mathcal{X})$, recall that the \emph{relative entropy} of $\nu$ with respect to $\mu$ is defined as
\begin{equation}\label{relent}
  H(\nu \| \mu) \doteq \left\{\begin{array}{ll} \int_\mathbb{R} \log (\tfrac{d\nu}{d\mu}) \, d\nu & \text{ if } \nu \ll \mu,\\
  +\infty & \text{ else,}  
\end{array}
\right.
\end{equation}
where $\nu \ll \mu$ denotes that $\nu$ is absolutely continuous with respect to $\mu$. For $q\in [1,\infty)$, define the $q$-th absolute moment map,
\begin{equation}\label{mqdef}
  m_q(\nu) \doteq \int_\R |x|^q \,\nu(dx), \quad \nu\in \mathcal{P}(\R).
\end{equation}
For $q=\infty$, let 
\begin{equation*}
m_\infty(\nu) \doteq \inf\{a>0: \nu([-a,a]^c) = 0\}, \quad \nu\in \mathcal{P}(\R).
\end{equation*}
For $p\in[1,\infty]$, define $\mathbb{H}_p:\mathcal{P}(\mathbb{R})\ra[0,\infty]$ to be a modification of the relative entropy with respect to $\mu_p$, perturbed by some $p$-th moment penalty,
 \begin{equation} \label{hpdefn}
  \mathbb{H}_p(\nu) \doteq \left\{ \begin{array}{ll}
 H(\nu \| \mu_p) + \tfrac{1}{p} (1 - m_p(\nu)) & \textnormal{ if } m_p(\nu) \le 1, \\
  +\infty & \textnormal{ else,} 
 \end{array}\right. 
  \end{equation}
where we take the convention $\frac{1}{\infty}=0$. We will show in Propositions \ref{prop-sanovcone} and \ref{prop-sanovsurface} that $(\lnp)_{n\in\N}$ satisfies an LDP with rate function $\mathbb{H}_p$.

Our LDP holds with respect to the so-called $q$-Wasserstein topologies on $\mathcal{P}_q(\R)$, introduced below. We refer to \cite[\S6]{villani2008optimal} for an extensive review of the Wasserstein topology.

\begin{definition}\label{def-wass}
Let $q\in[1,\infty]$, and let 
\begin{equation*}
  \mathcal{P}_q(\R) \doteq \{\mu \in \mathcal{P}(\R) : m_q(\mu) < \infty\}.
\end{equation*}
We equip this space with the $q$-\emph{Wasserstein} topology: a sequence $(\mu_n)_{n\in\N}\subset\mathcal{P}_q(\R)$ is said to converge to $\mu\in \mathcal{P}_q(\R)$ (with respect to the $q$-Wasserstein topology) if $\mu_n$ converges weakly to $\mu$ (denoted by $\mu_n\Rightarrow \mu$) and $m_q(\mu_n)\rightarrow m_q(\mu)$ as $n\rightarrow\infty$.
\end{definition}

\begin{remark}\label{rmk-wasswhy}
We consider the $q$-Wasserstein topology on probability measures (instead of, e.g., the weak topology or the $\tau$ topology) because we must consider a topology that is weak enough to allow an LDP to hold, but at the same time strong enough to allow certain moments to be continuous functionals of measures. In particular, we require a topology strong enough such that the moment map $m_q$ is continuous for $q<p$, which is used in the proofs of both the Gibbs conditioning result of Proposition \ref{prop-clln}  and the variational formula of \cite{gkr2} that relates ``quenched" and ``annealed" LDP rate functions for random projections of $\ell^p$ balls.
\end{remark}

\subsection{LDP under the cone measure}\label{sec-empcone}

In this section, we establish an LDP for $(\lnp)_{n\in\N}$ assuming $\anp$ is distributed according to the cone measure on $\snp$ defined below.

\begin{definition} \label{def-cone}
Let $\text{vol}_n(\cdot)$ denote the Lebesgue measure on $\R^n$, and let  $\gamma_{n,p}$ denote the \emph{cone measure} on $\snp$,
\begin{equation*}
\gamma_{n,p}(S)  \doteq \frac{ \text{vol}_n\left( \{ cx: x\in S, c\in[0,1]\}\right) } {\text{vol}_n ( \mathbb{B}_{n,p} )},
\end{equation*}
for $S$ a Borel measurable subset of $\snp$.
\end{definition}

\begin{proposition}[LDP under the cone measure]\label{prop-sanovcone}
Let $p\in[1,\infty]$ and assume $\anp \sim \gamma_{n,p}$. For $q<p$, the sequence of empirical measures $(\lnp)_{n\in \N}$ of \eqref{empir} satisfies an LDP in $\mathcal{P}_q(\R)$ with the strictly convex good rate function $\mathbb{H}_p$ of \eqref{hpdefn}.
\end{proposition}

The proof of Proposition \ref{prop-sanovcone} is given at the end of this subsection.

\begin{remark}
The primary obstacle in the proof of Proposition \ref{prop-sanovcone} (which differentiates it from the classical i.i.d.\ setting of Sanov's theorem) is that the coordinates $\anp_i$, $i=1,\dots, n$, are \emph{dependent}, by virtue of the fact that the vector $\anp$ is constrained to lie in the sphere $\snp$. Although LDPs for empirical measures of random vectors with other dependency structures have been established (such as \cite[\S6.3-6.6]{DemZeibook} for Markov chains or stationary sequences satisfying certain mixing conditions, or  \cite{arous1997large} for empirical spectral measures of random matrices), none of these apply in our setting.
\end{remark}

 \begin{remark}
It is natural to try to establish an LDP for $(\lnp)_{n\in\N}$ by
exploiting the fact that the sequence $(\alp^{(1,p)},
\alp^{(2,p)},\dots)$ can be viewed as an infinite triangular array,
where each row $\anp$, for $n\in \N$, is an exchangeable
vector. However, such an attempt fails. Indeed, on pg.\ 653 of
\cite{trashorras2002large}, which contains general large deviations
results for the empirical measure of a row-wise exchangeable
triangular array, it is stated that ``Even in the simple case of binary valued finite exchangeable random variables there is no general result for
the LD behavior of [the row-wise empirical measure]." The aforementioned quote is in reference to the result \cite[Theorem 2]{trashorras2002large}. That is, an exchangeable structure on its own is not sufficient for a general LDP result (and in particular, not sufficient for an LDP for $(\lnp)_{n\in\N}$). Instead, we take a different route, which relies on the following preliminary results.
\end{remark}

The analysis of the LDP for $(\lnp)_{n\in\N}$   is facilitated by the following probabilistic representation for  cone measures:

\begin{lemma}[{\cite[\S3]{rachev1991approximate} and \cite[Lemma 1]{schechtman1990volume}}] \label{lem-lprep}
Fix $n\in \N$ and $p\in[1,\infty]$. Suppose $X^{(n,p)} \sim \gamma_{n,p}$, and let $\ynp\sim \mu_p^{\otimes n}$. Then,
\begin{equation}\label{alphrep}
  \anp \eqdist \frac{Y^{(n,p)}}{\|Y^{(n,p)}\|_{n,p}}.
\end{equation}
\end{lemma}

Some extensions of Lemma \ref{lem-lprep} can also be found in \cite{barthe2005probabilistic}. The preceding representation is used in the proof of the following result, which generalizes a result from \cite[Theorem 6.6]{arous2001aging} beyond the special case where $p=2$ (for which the cone and surface measures coincide, and the surface measure can be expressed in terms of i.i.d.\ Gaussians).

\begin{proposition}\label{prop-weak}
Let $p\in[1,\infty]$ and assume $\anp \sim \gamma_{n,p}$. The sequence $(\lnp)_{n\in N}$ of \eqref{empir} satisfies an LDP in $\mathcal{P}(\R)$ equipped with the weak topology, with the good rate function $\mathbb{H}_p$ of \eqref{hpdefn}.
\end{proposition}

To prove Proposition \ref{prop-weak}, we appeal to the following representation of relative entropy. 

\begin{lemma}[Donsker-Varadhan; see, e.g., Lemma 1.4.3(a) of \cite{dupuis2011weak}] \label{lem-donvar}
Let $C_b(\mathcal{X})$ denote the space of bounded continuous functions from $\mathcal{X}$ into $\mathbb{R}$. Then, for all $\nu,\mu \in \mathcal{P}(\mathcal{X})$, we have
\begin{equation*}
  H(\nu \| \mu) = \sup_{f\in C_b(\R)} \left\{ \int_{\mathcal{X}} f d\nu - \log\int_{\mathcal{X}} e^{f} d\mu \right\}.
\end{equation*}

\end{lemma}

\begin{proof}[Proof of Proposition \ref{prop-weak}]
Let $\ynp \sim \mu_p^{\otimes n}$ be as in Lemma \ref{lem-lprep}, and let $\lnpy$ denote the empirical measure of $\ynp$,
\begin{equation*}
  \lnpy \doteq \frac{1}{n}\sum_{i=1}^n \delta_{\ynp_i}.
\end{equation*} 
Since $(\lnpy)_{n\in\N}$ satisfies an LDP in $\mathcal{P}(\R)$ by Sanov's theorem (see, e.g., \cite[Sect. 6.2]{DemZeibook}), $(\lnpy)_{n\in\N}$ is also exponentially tight (recall Definition \ref{def-etight}, and see e.g., \cite[Lemma 2.6]{lynch1987large}). Furthermore, due to the finite exponential moment $\E[e^{t \|Y^{(n,p)}\|_{n,p}^p}] < \infty$ for $|t| < 1/p$, for $m_p$ as in \eqref{mqdef}, the sequence $(m_p(\lnpy))_{n\in\N}$  satisfies an LDP in $\R$ by Cram\'er's theorem (see, e.g., \cite[Sect. 2.2.2]{DemZeibook}), and hence $(m_p(\lnpy))_{n\in\N}$ is also exponentially tight. As a consequence, the joint sequence  $(\lnpy, m_p(\lnpy))_{n\in \N}$ is also exponentially tight, with respect to the product topology on $\mathcal{P}(\R) \times \R_+$.

 We claim (and prove below) that the joint sequence $(\lnpy, m_p(\lnpy))_{n\in \N}$ satisfies an LDP in $\mathcal{P}(\R)\times\mathbb{R}_+$ with the convex good rate function $J$ defined as follows:  for $\lambda \in \mathcal{P}(\R)$ and $c \in \R_+$, let
\begin{equation}\label{jointj}
 J(\lambda,c) \doteq \sup_{f\in C_b(\R),\, t\in \R} \left\{ \int_\R f(y) \, \lambda(dy) + tc - \log \int_\R e^{f(y) + t|y|^p} \mu_p(dy)   \right\}
 \end{equation}
 where $C_b(\R)$ denotes the space of all bounded continuous functions from $\R$ to $\R$. Due to Cram\'er's theorem for general Polish spaces (see, e.g.,   \cite[Theorem 6.1.3]{DemZeibook}), the sequence $(\lnpy, m_p(\lnpy))_{n\in \N}$ satisfies a \emph{weak large deviation principle} in the space $\mathcal{P}(\R) \times \R_+$, with rate function $J$ of \eqref{jointj}; that is, the LDP upper bound in Definition \ref{def-ldp} holds only for compact sets (see \cite[pg.7]{DemZeibook} for a precise definition of a weak large deviation principle). The full LDP with good rate function $J$ then follows as a consequence of the weak LDP combined with the previously established exponential tightness (see, e.g., \cite[Lemma 1.2.18]{DemZeibook}).

We now derive an alternative form of the rate function $J$ defined in \eqref{jointj}. Note that for $t \ge \frac{1}{p}$, we have $\log \int_\R e^{f(y) + t|y|^p} \mu_p(dy) = +\infty$. For $t < 1/p$, define $\tilde{\mu}_{p,t} (dy) \doteq \frac{1}{Z_{p,t}} e^{t |y|^p} \mu_p(dy)$, where $Z_{p,t}\doteq \int_\R e^{t|y|^p}\mu_p(dy)$ is a normalizing constant so that $\tilde{\mu}_{p,t}$ is a probability measure. Note that $\tilde{\mu}_{p,t}$ is well defined for $t < \frac{1}{p}$, due to the definition of $\mu_p$ from \eqref{gengsn}.

 For $t< \frac{1}{p}$, apply the Donsker-Varadhan variational formula for relative entropy (as stated in Lemma \ref{lem-donvar}) to the bounded continuous function $f$ and the probability measure $\tilde{\mu}_{p,t}$. We then find that for $(\lambda,c)\in\mathcal{P}(\R) \times \R_+$,
\begin{align*}
  J(\lambda, c) &= \sup_{t < 1/p} \left\{ tc + H(\lambda \| \tilde{\mu}_{p,t}) - \log Z_{p,t}  \right\}\\
    &= H(\lambda \| \mu_p) + \sup_{t < 1/p} \left\{ tc + \int_\R \log\frac{d\mu_p}{d\tilde{\mu}_{p,t}} \lambda(dy) - \log Z_{p,t}  \right\}\\
    &= H(\lambda \| \mu_p) + \sup_{t < 1/p} \left\{ tc - \int_\R t|y|^p \lambda(dy) + \log Z_{p,t} - \log Z_{p,t}  \right\}\\
    &=  H(\lambda \| \mu_p ) + \sup_{ t < 1/p} \left\{ tc  - t\,m_p(\lambda) \right\} \\
  &= \left\{ \begin{array}{ll}
 H(\lambda \| \mu_p) + \tfrac{1}{p} (c - m_p(\lambda))  & \textnormal{ if } m_p(\lambda) \le c,\\
  +\infty & \textnormal{ else.} \end{array}\right.
\end{align*}
For $(\lambda,c)\in \mathcal{P}(\R)\times \R_+$, define $G_p(\lambda,c) \doteq \lambda( \,\cdot \, \times c^{1/p})$. An elementary consequence of Slutsky's theorem is that $G_p:\mathcal{P}(\R)\times\R_+ \ra \mathcal{P}(\R)$  is continuous (w.r.t.\ the weak topology on $\mathcal{P}(\R)$ and the induced product topology on $\mathcal{P}(\R)\times \R_+$). By \eqref{empir} and Lemma \ref{lem-lprep},  $ \lnp$ is equal in distribution to $G_p( \lnpy, m_p(\lnpy))$, so the contraction principle (see, e.g., \cite[Theorem 4.2.1]{DemZeibook}) yields an LDP for $(\lnp)_{n\in\N}$ with the good rate function
\begin{align*}
  \mathbb{J}_p(\nu) &\doteq \inf \left\{  J(\lambda, c) : \lambda \in \mathcal{P}(\mathbb{R}), c \in \mathbb{R}_+ , G_p(\lambda,c) = \nu  \right\}\\
 &= \left[ \begin{array}{ll}  \inf_{c\ge 0}\left\{ 
 H(\nu(\,\cdot\, \times c^{-1/p} )\, \| \, \mu_p) + \tfrac{1}{p} (c - c\,m_p(\nu)) \right\} & \textnormal{ if } m_p(\nu) \le 1,\\
  +\infty & \textnormal{ else.}
 \end{array}\right.
 \end{align*}
If $m_p(\nu)\le 1$, then
\begin{align*}
\mathbb{J}_p(\nu) 
  &= \inf_{c\ge 0} \left\{ H(\nu \| \mu_p(\cdot \times c^{1/p}))  +\tfrac{c}{p} (1- m_p(\nu)) \right\} \\
  &= H(\nu \| \mu_p) + \inf_{c\ge 0} \left\{   -\tfrac{1}{p}\log c - \tfrac{1-c}{p} m_p(\nu) + \tfrac{c}{p} (1- m_p(\nu)) \right\} = \mathbb{H}_p(\nu).
\end{align*}\end{proof}

\begin{remark}
Note that the above proof of Proposition \ref{prop-weak} relies on an LDP for the joint sequence $(\lnpy, m_p(\lnpy))_{n\in \N}$. An alternative approach to the one presented would be to appeal to results on the ``partial large deviation principle" established for self-normalized processes, as stated in \cite[Theorem 1.1]{dembo1998self}.	
\end{remark}

\begin{remark}
Note that Proposition \ref{prop-weak} cannot be obtained via a direct application of the contraction principle to the map $\nu\mapsto G_p(\nu,m_p(\nu))$. Indeed, there does not appear to be a standard topology on $\mathcal{P}(\R)$ such that both the sequence $(\lnpy)_{n\in\N}$ satisfies an LDP with good rate function \emph{and} the map $m_p(\cdot)$ is continuous with respect to that topology. In particular,  $m_p(\cdot)$ is continuous with respect to the Wasserstein-$r$ topology if and only if $r\ge p$; on the other hand, \cite{LeoNaj02} and \cite{wang2010sanov} show that the sequence $(\lnpy)_{n\in\N}$ satisfies an LDP with good rate function with respect to the Wasserstein-$r$ topology if and only if $r<p$.
 \end{remark}

\begin{lemma}\label{lem-compact}
For $p\in[1,\infty]$, let $K_p \doteq \{ \nu \in \mathcal{P}(\mathbb{R}) : m_p(\nu) \le 1 \}$. Then, for all $q <p$, the set $K_p\subset\mathcal{P}_q(\R)$ is compact in $\mathcal{P}_q(\R)$. In addition, $K_p$ is convex and non-empty.
\end{lemma}

\begin{proof}
The convexity and non-emptiness of $K_p$ are elementary. As for compactness, we first prove that $K_p$ is weakly compact in $\mathcal{P}(\R)$. For $\nu \in K_p$, for all $M > 0$, Chebyshev's inequality yields  $\nu([-M,M]^c) \le m_p(\nu)/ M^p \le 1/{M^p}$, so $K_p$ is tight, and by Prokhorov's theorem, precompact. Note that $K_p$ is weakly closed in $\mathcal{P}(\R)$, since it is a level set of the lower semicontinuous map $m_p$. Therefore, $K_p$ is weakly compact.

To verify Wasserstein compactness, it suffices to show that the set of probability measures in $K_p$ have uniformly integrable $q$-th moments \cite[Proposition 7.1.5]{ambrosio2008gradient}. This latter condition follows from the de la Vall\'ee-Poussin criterion (see, e.g., \cite[II.22]{dellacherie2011probabilities}) since for $g(x) \doteq |x|^{p/q}$ (which satisfies the superlinear growth condition $\lim_{|t|\ra\infty}g(t)/t =\infty$), we have
\begin{align*}
\sup_{\nu \in K_p} \int_\mathbb{R} g(|x|^q) \nu(dx) &= \sup_{\nu \in K_p} m_p(\nu) \le 1 <\infty.
\end{align*}
\end{proof}

\label{pf-cone}
\begin{proof}[Proof of Proposition \ref{prop-sanovcone}]
Fix $p\in[1,\infty]$ and $q<p$. In view of the fact that $(\lnp)_{n\in\N}$ satisfies an LDP in $\mathcal{P}(\R)$ with respect to the weak topology due to Proposition \ref{prop-weak}, in order to establish the LDP for $(\lnp)_{n\in\N}$ in $\mathcal{P}_q(\R)$ with respect to the $q$-Wasserstein topology, it suffices to show exponential tightness of $(\lnp)_{n\in\N}$ in $\mathcal{P}_q(\R)$ (see \cite[Corollary 4.2.6]{DemZeibook}).  Let $K_p$ be the compact set  defined in Lemma \ref{lem-compact}. Note that $m_p(\lnp) = 1$ a.s., so $\P(\lnp \in K_p^c) = 0$ and hence $\log \P(\lnp \in K_p^c)  = -\infty$ for all $n\in\N$, implying exponential tightness of  $(\lnp)_{n\in\mathbb{N}}$.

The strict convexity of $\mathbb{H}_p$ follows from the strict convexity of the relative entropy $H(\cdot \| \mu_p)$ and the linearity of the $p$-th moment $m_p$.
\end{proof}

\subsection{LDP under the surface measure}\label{sec-surface}

In this section, we establish an LDP for $(\lnp)_{n\in\N}$ assuming $\anp$ is distributed according to the surface measure on $\snp$, which is defined below.

\begin{definition} \label{def-surf}
Let $\text{area}_{n,p}(\cdot)$ denote the (unnormalized) surface area measure on $\snp$, and let $\sigma_{n,p}$ denote the (normalized) \emph{surface measure} on $\snp$, 
\begin{equation*}
  \sigma_{n,p}(S) \doteq \frac{\text{area}_{n,p}(S)}{\text{area}_{n,p}(\snp)}, 
\end{equation*}
for $S$ a Borel measurable subset of $\snp$.
\end{definition}


\begin{proposition}[LDP under the surface measure]\label{prop-sanovsurface}
Let $p\in[1,\infty]$ and assume $\anp \sim \sigma_{n,p}$. For $q<p$, the sequence of empirical measures $(\lnp)_{n\in N}$ of \eqref{empir} satisfies an LDP in $\mathcal{P}_q(\R)$  with the strictly convex good rate function $\mathbb{H}_p$ of \eqref{hpdefn}.
\end{proposition}

The proof of Proposition \ref{prop-sanovsurface} is deferred to the
end of this subsection, after first establishing some preliminary lemmas. 

Note that the cone measure coincides with the surface measure on
$\snp$ if and only if $p=1,2$, or $\infty$. See
\cite[\S3]{rachev1991approximate}  and \cite[\S3]{naor2007surface} for
more extensive discussions. More generally, we have the following
relationship between the cone and surface measures.  

\begin{lemma}[{\cite[Lemma 2]{naor2003projecting}}]\label{lem-surfconedens}
Let $p\in[1,\infty)$. Then, 
\begin{equation*}
  \frac{d\sigma_{n,p}}{d\gamma_{n,p}}(x) = C_{n,p} \left(\sum_{i=1}^n |x_i|^{2p-2}\right)^{1/2}, \quad x\in \snp,
\end{equation*}
where $C_{n,p}$ is a normalizing constant.
\end{lemma}

Next, we state a general result about LDPs for two sequences of measures that satisfy a particular absolute continuity relation.

\begin{lemma}\label{lem-equivldp}
Let $\mathcal{X}$ be a Polish space, and for $n\in \N$, let $\psi_n:\mathcal{X}\rightarrow\R$. Suppose that there exists a sequence of finite constants $(M_n)_{n\in \N}$ satisfying $  \lim_{n\ra\infty} \tfrac{M_n}{n} = 0$, such that,
\begin{equation}\label{psibd}
  |\psi_n(\lambda)| \le M_n, \quad \mu_n\textnormal{-a.e. } \lambda\in \mathcal{X}, \quad \forall n\in \N.
\end{equation}
Let $(\mu_n)_{n\in\N}\subset\mathcal{P}(\mathcal{X})$ be a sequence of probability measures that satisfies an LDP with a good rate function $\mathbb{I}(\cdot)$. Define 
\begin{equation}\label{nundef}
  \nu_n(d\lambda) \doteq \tfrac{1}{\int_{\mathcal{X}} e^{\psi_n(\lambda)} \mu_n(d\lambda)} e^{\psi_n(\lambda)} \mu_n(d\lambda).
\end{equation}
Then, $(\nu_n)_{n\in \N}$ satisfies an LDP with the same good rate function $\mathbb{I}(\cdot)$.
\end{lemma}
\begin{proof}
For $\phi:\mathcal{X}\rightarrow \R$ continuous and bounded, define
\begin{equation*}
  \Lambda_\phi \doteq \lim_{n\rightarrow\infty} \frac{1}{n}\log\int_{\mathcal{X}} e^{n\,\phi(\lambda)} \mu_n(d\lambda),
\end{equation*}
where the limit exists due to  the LDP for $(\mu_n)_{n\in \N}$ and Varadhan's Lemma (see, e.g., \cite[Theorem 4.3.1]{DemZeibook}). The definition of $\nu_n$ in \eqref{nundef}, the bound on $\psi_n$ for all $n\in\N$ given in  \eqref{psibd}, and the assumption that $M_n/n \ra 0$ imply that 
\begin{align*}
\liminf_{n\ra\infty}  \frac{1}{n}\log\int_{\mathcal{X}} e^{n\,\phi(\lambda)} \nu_n(d\lambda) &\ge   \liminf_{n\rightarrow\infty} \frac{1}{n}\log \int_{\mathcal{X}} e^{n\,\phi(\lambda) -M_n} \mu_n(d\lambda)\\
 &= \Lambda_\phi -\lim_{n\rightarrow\infty} \frac{M_n}{n} = \Lambda_\phi,\\
\limsup_{n\ra\infty}  \frac{1}{n}\log\int_{\mathcal{X}} e^{n\,\phi(\lambda)} \nu_n(d\lambda)  &\le  \limsup_{n\rightarrow\infty} \frac{1}{n}\log \int_{\mathcal{X}} e^{n\,\phi(\lambda) +M_n} \mu_n(d\lambda)\\
 &= \Lambda_\phi +\lim_{n\rightarrow\infty} \frac{M_n}{n} = \Lambda_\phi.
\end{align*}
Thus, we have shown that
\begin{equation*}
  \Lambda_\phi = \tilde\Lambda_\phi \doteq \lim_{n\ra\infty} \frac{1}{n}\log\int_{\mathcal{X}} e^{n\,\phi(\lambda)} \nu_n(d\lambda).
\end{equation*}

Note that $(\mu_n)_{n\in\N}$ is exponentially tight since it satisfies an LDP with good rate function and $\mathcal{X}$ is Polish \cite[Lemma 2.6]{lynch1987large}. We claim that $(\nu_n)_{n\in\N}$ is also exponentially tight. Let $L < \infty$, and let $K_L\subset\mathcal{X}$ be a compact set such that 
\begin{equation}
\limsup_{n\ra\infty} \frac{1}{n}\log\mu_n(K_L^c) < -L. \label{exptight}
\end{equation}
Then, given $L < \infty$, note that
\begin{align*}
\log\nu_n( K_L^c) &=  \log \int_{ K_L^c} e^{\psi_n(\lambda)} \mu_n(d\lambda) - \log \int_{\mathcal{X}} e^{\psi_n(\lambda)} \mu_n(d\lambda)\\
   &\le M_n + \log \mu_n(K_L^c) + M_n - \log \mu_n(\mathcal{X}).
\end{align*}
Taking the limit supremum as $n\ra\infty$, using \eqref{exptight}, and applying the fact that $\frac{M_n}{n} \ra 0$, we find that
\begin{equation*}
\limsup_{n\ra\infty} \frac{1}{n} \log \nu_n({K}_L^c) \le \limsup_{n\ra \infty} \frac{1}{n} \log\mu_n(K_L^c) < -L.
\end{equation*}
Since $\Lambda_\phi = \tilde\Lambda_\phi$, and both sequences $(\mu_n)_{n\in \N}$ and $(\nu_n)_{n\in \N}$ are exponentially tight, Bryc's inverse lemma (see, e.g.,  \cite[Theorem 4.4.2]{DemZeibook}) implies that the two sequences satisfy the same LDP with the good rate function $\mathbb{I}(\lambda) = \sup_{\phi\in C_b(\mathcal{X})}\{\phi(\lambda) - \Lambda_\phi\}$.
\end{proof}

We apply the preceding lemma to the absolute continuity relation of Lemma \ref{lem-surfconedens} to prove the LDP under the surface measure.

\label{pf-surf}
\begin{proof}[Proof of Proposition \ref{prop-sanovsurface}]
For $p=1,2,\infty$, we have $\sigma_{n,p} = \gamma_{n,p}$, and so the result trivially follows from Proposition \ref{prop-sanovcone}. Therefore, we restrict to $p \in (1,2)$ or $p\in(2,\infty)$. In this setting, we apply Lemma \ref{lem-equivldp} to the following: let $\mathcal{L}_n:\R^n\ra \mathcal{P}(\R)$ be the map that takes a vector $a\in\R^n$ to the measure $\frac{1}{n}\sum_{i=1}^n\delta_{a_i}\in \mathcal{P}(\R)$, and set
\begin{itemize}
\item $\mathcal{X} = \mathcal{P}_q(\R)$ for some $q< p$,
\item $\mu_n=\gamma_{n,p} \circ \mathcal{L}_n^{-1}$ and $\nu_n = \sigma_{n,p}\circ \mathcal{L}_n^{-1}$,
\item $\psi_n = \psi = \tfrac{1}{2}\log m_{2p-2}$,
\item $M_n \doteq |\frac{1}{2} - \frac{1}{p}| \log n$.
\end{itemize}
The choice of $(M_n)_{n\in\N}$ will become clear with the proof. We show that our setup satisfies the hypotheses of Lemma \ref{lem-equivldp}. For $n\in \N$, let 
\begin{equation*}
  A_n \doteq \left\{\frac{1}{n}\sum_{i=1}^n \delta_{a_i} : a\in \R^n\right\} \subset \mathcal{P}(\R).
\end{equation*}
Note that for all $n\in \N$, $\mu_n(A_n) = 1$. For $\lambda\in A_n$, using Lemma \ref{lem-surfconedens}, we can write the Radon-Nikodym derivative of $\nu_n$ with respect to $\mu_n$ (defined on a set of $\mu_n$ measure 1), as 
\begin{equation*}
  \frac{d\nu_n}{d\mu_n}(\lambda) = \frac{d\sigma_{n,p}}{d\gamma_{n,p}}(\mathcal{L}_n^{-1} \lambda) = C_{n,p} \left(n\, m_{2p-2}(\lambda)\right)^{1/2} = n^{1/2}C_{n,p}\, e^{\psi(\lambda)}.
\end{equation*}
We know from Proposition \ref{prop-sanovcone} that $(\mu_n)_{n\in\N}$ satisfies an LDP in $\mathcal{P}_q(\R)$ (for $q<p$), with good rate function $\mathbb{H}_p$. 

As for the boundedness properties of $\psi_n= \psi =  \frac{1}{2} \log m_{2p-2}$ stipulated in Lemma \ref{lem-equivldp}, first note that due to H\"older's inequality, for any $0 < r < s < \infty$ and $\lambda \in \mathcal{P}(\R)$, we have  $m_r(\lambda) \le [m_s(\lambda)]^{r/s}$. Now, fix $1 < p < 2$. Then $0 < 2p - 2 < p$, and so applying the preceding  inequality with $r = 2p - 2$ and $s = p$, we see that for $\lambda \in A_n$,
\begin{equation*}
  \mu_n( m_{2p-2}(\lambda) > 1) \le \mu_n( m_p(\lambda)^{2-(2/p)} > 1)  = \P\left(  \| \anp\|_{n,p} > 1 \right)  = 0,
\end{equation*}
where the last equality follows since $X^{(n,p)} \in \snp$, $\P$-a.s. On the other hand, to lower bound $m_{2p-2}(\lambda)$, recall that for $0 < r < s < \infty$ and $a\in \R^n$, we have $\|a\|_{n,s} \le \|a\|_{n,r}$. Applying this inequality with $r=2p-2$ and $s=p$,  and recalling $\lnp$ from \eqref{empir}, we have $\P$-a.s.,
\begin{equation*}
 m_{2p-2}(\lnp)  \,=\, n^{1-(2/p)} \|\anp\|_{n,2p-2}^{2p-2}    \, \ge \, n^{1-(2/p)} \|\anp\|_{n,p}^{2p-2}   \,=\, n^{1-(2/p)},
\end{equation*}
where the last equality once again uses the fact that $\P$-a.s. (under the cone measure), $X^{(n,p)}$ lies on the unit $\ell^p$ sphere $\snp$. In a similar fashion, for $2 < p < \infty$, we can show $\mu_n( m_{2p-2}(\lambda) < 1)=0$ for $\lambda \in A_n$, and $ m_{2p-2}(\lnp) \le n^{1-(2/p)}$, $\P$-a.s. 

Since $\mu_n(A_n) =1$, we have shown that for $\mu_n$-a.e.\ $\lambda \in \mathcal{P}(\R)$,
\begin{equation*}
  |\psi(\lambda)| = \left|\tfrac{1}{2}\log m_{2p-2}(\lambda)\right| \le M_n,
\end{equation*}
where $M_n \doteq \left|\tfrac{1}{2} - \tfrac{1}{p}\right| \log n$. Since $M_n/n\ra 0$, we can apply Lemma \ref{lem-equivldp}, which shows that the sequence of empirical measures $(\lnp)_{n\in\N}$ under the surface measure $\sigma_{n,p}$ satisfies an LDP with the same good rate function as under the cone measure $\gamma_{n,p}$.
\end{proof}


\section{Gibbs conditioning} \label{sec-gibbscond}

In this section, we use the LDPs of Propositions \ref{prop-sanovcone} and \ref{prop-sanovsurface} to prove our main result Proposition \ref{prop-clln}, which states a conditional limit law: that is, the asymptotically most likely distribution of a sequence of random variables, conditional on a large deviation. In the novel setting of $\ell^p$ spheres, we are able to lend  a geometric meaning to this conditional limit law.

 To state our result, we first introduce some additional notation. Let $h(\cdot)$ denote the differential entropy of a probability measure $\nu\in\mathcal{P}(\R)$,
\begin{equation}
h(\nu) \doteq \left\{ \begin{array}{ll}
 -\displaystyle\int_\R \tfrac{d\nu}{dx} \log \left(\tfrac{d\nu}{dx}\right)\, dx & \text{ if $\nu \ll \textnormal{Lebesgue measure on } \R$,}\\
 -\infty & \text{ else.}
 \end{array}\right.
\end{equation}
Differential entropy arises naturally in the analysis of the rate function $\mathbb{H}_p$, since for measures $\nu\in\mathcal{P}(\R)$ that are absolutely continuous with respect to Lebesgue measure and satisfy $m_p(\nu)\le 1$, we can rewrite the function $\mathbb{H}_p$ defined in \eqref{hpdefn} as:
\begin{equation}
\mathbb{H}_p(\nu) = -h(\nu)  + \log(2p^{1/p}\Gamma(1+\tfrac{1}{p})) + \tfrac{1}{p}. \label{hentrep}
\end{equation}
We obtain the following result as a consequence of the LDP of Theorem \ref{prop-sanovcone}, which involves a constrained maximum entropy problem. In particular, we crucially use the fact that the LDP for $(\lnp)_{n\in\N}$ holds with respect to the Wasserstein topology, which is stronger than the usual weak topology.

\begin{proposition}\label{prop-clln}
Fix $p\in[1,\infty]$ and suppose that either 
$X^{(n,p)}\sim \sigma_{n,p}$ or $X^{(n,p)}\sim \gamma_{n,p}$. Fix a closed interval $C=[\alpha,\beta]\subset \R$, and for $\epsilon > 0$, let $C_\epsilon\doteq [\alpha-\epsilon,\beta+\epsilon]$.  Then,  for $q < p$, the optimizing measure
\begin{equation}\label{nustar}
  \nu_* \doteq \arg \max \{ h(\nu) : m_p(\nu) \le 1, m_q(\nu) \in C \}
\end{equation}
is well defined (i.e., exists and is unique), and 
\begin{equation}\label{clln1}
  \lim_{\epsilon \ra 0}\, \lim_{n\ra \infty} \P(\lnp \in \cdot \,\, |\, m_q(\lnp) \in C_\epsilon) = \delta_{\nu_*}.
\end{equation}
Moreover, for $k\in \N$,
\begin{equation}\label{clln2}
  \P\left(\left. n^{1/p}(X_1^{(n,p)},\dots,X_k^{(n,p)}) \in \cdot \right| \,\,\tfrac{1}{n} \|n^{1/p}X^{(n,p)}\|_{n,q}^q \in C_\epsilon\right) \Rightarrow  \nu_*^{\otimes k},
\end{equation}
as $n\ra\infty$ followed by $\epsilon \ra 0$.
\end{proposition}

The proof of Proposition \ref{prop-clln} relies crucially on the so-called ``Gibbs conditioning principle" (see, e.g., \cite{csiszar1984sanov, dembo1996refinements, stroock1991microcanonical}), stated precisely below.

\begin{proposition}[{Gibbs conditioning principle, \cite[Theorem 7.1]{leonard2010entropic}}]\label{th-gibbs}
Let $\mathcal{X}$ be a topological space, and let $(\xi_n)_{n\in \N}$ be a sequence of $\mathcal{X}$-valued random variables that satisfies an LDP with good rate function $\mathbb{I}$. In addition, let $F\subset\mathcal{X}$  be a subset such that
\begin{enumerate}
\item $\mathbb{I}(F)\doteq \inf_{x\in F} \mathbb{I}(x) < \infty$;
\item $F$ is closed;
\item $F = \bigcap_{\epsilon > 0} F_\epsilon$ for a family of sets $(F_\epsilon)_{\epsilon > 0}$ such that $F_\epsilon \subset\mathcal{X}$ for all $\epsilon > 0$ and $\P(\xi_n \in F_\epsilon) > 0$ for all $\epsilon > 0$ and $n\in\N$;
\item $F \subset (F_\epsilon)^\circ$ for all $\epsilon > 0$.
\end{enumerate}
Let $\mathcal{M}_F$ be the set of $x\in F$ which minimize $\mathbb{I}$. That is,
\begin{equation*}
  \mathcal{M}_F \doteq \{x\in F : \mathbb{I}(x) = \mathbb{I}(F)\}.
\end{equation*}
Then, for all open $G\subset\mathcal{X}$ such that $\mathcal{M}_F\subset G$, we have
\begin{equation}\label{res1gc}
  \limsup_{\epsilon\ra 0}\limsup_{n\ra\infty}\frac{1}{n} \log\P(\xi_n\not\in G\,|\, \xi_n \in F_\epsilon) < 0.
\end{equation}
As a consequence, if $\mathcal{M}_F=\{\bar x\}$ is a singleton, then
\begin{equation}\label{res2gc}
  \lim_{\epsilon\ra 0} \lim_{n\ra\infty} \P(\xi_n \in \cdot | \xi_n \in F_\epsilon) = \delta_{\bar x} .
\end{equation}
\end{proposition}

\begin{proof}[Proof of Proposition \ref{prop-clln}] 
Fix $p\in[1,\infty]$ and $q < p$, and recall that due to Propositions \ref{prop-sanovcone} and \ref{prop-sanovsurface}, the sequence $(\lnp)_{n\in\N}$ satisfies an LDP in $\mathcal{P}_q(\R)$ (equipped with the $q$-Wasserstein topology) with strictly convex good rate function $\mathbb{H}_p$.

We first  show that Proposition \ref{th-gibbs} applies in the setting of Proposition \ref{prop-clln}. Fix $C=[\alpha,\beta]\subset\R$ and for $\epsilon > 0$, let $C_\epsilon = [\alpha-\epsilon,\beta+\epsilon]$. The set $F_\epsilon \doteq \{\nu\in\mathcal{P}_q(\R) : m_q(\nu) \in C_\epsilon\}$ is closed due to the continuity of $m_q$ of \eqref{mqdef} with respect to the $q$-Wasserstein topology, and hence, as an intersection of closed sets, the set $ F \doteq \bigcap_{\epsilon > 0} F_\epsilon  = \{ \nu \in {\mathcal P}(\mathbb{R}): m_q (\nu) \in C \}$ is also closed. Moreover, for $\epsilon > 0$, we also have
\begin{equation*}
  F = m_q^{-1}(C) \subset m_q^{-1}(C_\epsilon^\circ) \subset [m_q^{-1}(C_\epsilon)]^\circ = (F_\epsilon)^\circ,
\end{equation*}
where the second inclusion again uses the continuity of $m_q$. Next, let us show that $\mathcal{M}_F$ is a singleton  (i.e., $\nu_*$ of \eqref{nustar} is well defined). Recall that
\begin{equation*}
    \mathcal{M}_F \doteq \left\{\nu \in F : \mathbb{H}_p(\nu) = \min_{\nu \in F} \mathbb{H}_p(\nu)\right\}.
\end{equation*}
Note that $F$ is closed and convex, since it is the preimage of a closed, convex set $C$ under a continuous linear map $m_q$. Because $\mathbb{H}_p$ is lower semicontinuous and has compact level sets, it attains its minimum within $F$. That the minimum is attained at a unique $\nu_* \in F$ follows from the strict convexity of $\mathbb{H}_p$. The ``maximum entropy"  representation for $\nu_*$ given in \eqref{nustar} follows from the expression for  $\mathbb{H}_p$ given in \eqref{hentrep}. Thus, the limit \eqref{clln1} follows from the limit \eqref{res2gc} of Proposition \ref{th-gibbs}, applied to $\xi_n = \mathcal{L}_{n,p}$.

As for the second result \eqref{clln2}, this follows from an elementary ``propagation of chaos" result \cite[Proposition 2.2]{sznitman1991topics} that, under the assumption of exchangeability, establishes the equivalence of statements like \eqref{clln1} (regarding convergence in probability of the empirical measure to a deterministic measure) and statements like \eqref{clln2} (regarding joint convergence in distribution of any fixed $k$ of the underlying random variables). \end{proof}

As a prerequisite for the proof of Theorem \ref{cor-asymp}, recall the following basic information-theoretic fact.

\begin{lemma}\label{lem-maxent}
Fix $r_i:\R\ra\R$, $\alpha_i\in \R$, for $i=1,\dots,m$, and  $s_j:\R\ra\R$, $\beta_j\in\R$, for $j=1,\dots, n$. Consider the following maximization problem:
\begin{equation}\label{maxent}
\begin{aligned}
& \underset{\nu\in\mathcal{P}(\R)}{\textnormal{maximize}}
& & h(\nu) \\
& \textnormal{subject to}
& & \int_\R r_i(x)\nu(dx) = \alpha_i \textnormal{ for } i=1,\dots, m,\\
& & & \int_\R s_j(x)\nu(dx) \le  \beta_j \textnormal{ for } j=1,\dots, n .
\end{aligned}
\end{equation}
Then, a probability measure $\nu_*\in\mathcal{P}(\R)$ attains the maximum in \eqref{maxent} if and only if it is of the following form:
\begin{equation*}
  \nu_*(dx) = \exp\left(-1-\kappa_0^* - \sum_{i=1}^m\lambda_i^*r_i(x) - \sum_{j=1}^n \mu_j^*s_j(x)  \right) \,dx,
\end{equation*}
with non-negative constants $\kappa_0^*$, $(\lambda_i^*)_{i=1}^m$, and $(\mu_j^*)_{j=1}^n$ chosen such that $\nu_*$ lies in $\mathcal{P}(\R)$ and satisfies the constraints in \eqref{maxent}. Moreover, if $\nu_*$ attains the maximum in \eqref{maxent}, then 
\begin{equation}\label{compslack}
\left(\int_\R s_j(x) \nu_*(dx) - \beta_j\right) \cdot \mu_j^* = 0
\end{equation}
for all $j=1,\dots, n$.
\end{lemma}

The preceding lemma is standard. See \cite[\S 12.1]{cover2001info} for a slight simplification of \eqref{maxent} with only equality constraints, or see \cite[Ex. 5.3]{boyd2004convex} for a version of \eqref{maxent} with discrete entropy and both equality and inequality constraints. The claim \eqref{compslack} is the so-called ``complementary slackness" condition  (see, e.g., \cite[\S5.5.2]{boyd2004convex}).

As a final preliminary, define the following family of probability measures for $q \in [1,\infty)$, $\beta > 0$:
\begin{equation}\label{mupb}
  \mu_{q,\beta}(dx) \doteq \frac{1}{2(\beta q)^{1/q}\Gamma(1+\tfrac{1}{q})} e^{-|x|^q/(\beta q)}dx, \quad x\in \R.
\end{equation}
Note that $\mu_{p,1}$ corresponds to $\mu_{p}$ of \eqref{gengsn}.

\begin{proof}[Proof of Theorem \ref{cor-asymp}]
For a random variable $X\sim \mu_q$, note that $\beta^{1/q} X \sim \mu_{q,\beta}$. Due to the unconditional limit theorem Proposition \ref{lem-pbm}, under either the surface measure or the cone measure, we have  $\lambda_{n,q}^{(\beta,k)}\Rightarrow \mu_{q,\beta}^{\otimes k}$, where $\beta \le \beta_{p,q}$ of \eqref{betapq}, and $\mu_{q,\beta}$ is as in \eqref{mupb}. Therefore, it suffices to show that under either the surface measure or the cone measure, the conditional distribution $\widehat{\lambda}_{n,p|q}^{(\beta+\epsilon,k)}$ also converges weakly to the same limit $\mu_{q,\beta}^{\otimes k}$. In view of \eqref{clln2} of Proposition \ref{prop-clln}, it suffices to show that when $C=[0,\beta]$, the unique maximizer $\nu_*$ of \eqref{nustar} satisfies $\nu_* = \mu_{q,\beta}$. Note that due to the maximum entropy calculations of Lemma \ref{lem-maxent}, we have for $\beta > 0$,
\begin{equation}\label{nustar2}
  \mu_{q,\beta} = \arg\max\{h(\nu) : m_q(\nu) \le \beta \}.
\end{equation}
To show that $\mu_{q,\beta} = \nu_*$, it suffices to show that $m_p(\mu_{q,\beta}) \le 1$. After some elementary calculations, we find that since $\beta \le \beta_{p,q}$, and using the expression for $\beta_{p,q}$ in \eqref{betapq},
\begin{equation*}
m_p(\mu_{q,\beta}) = \beta^{p/q} q^{p/q} \frac{\Gamma(\frac{p+1}{q})}{\Gamma(\frac{1}{q})} \le \frac{1}{q^{p/q}}\left(\frac{\Gamma(\tfrac{1}{q})}{\Gamma(\tfrac{p+1}{q})}\right)  q^{p/q}\frac{\Gamma(\frac{p+1}{q})}{\Gamma(\frac{1}{q})} = 1. \end{equation*}\end{proof}

\begin{remark}[``small" $\beta$]\label{rmk-beta}
The constant $\beta_{p,q}$ is chosen such that for $\beta \le \beta_{p,q}$ and the interval $C=[0,\beta]$, the variational problem of \eqref{nustar} has an explicit solution with a natural geometric interpretation. Note that the explicit solution to the simpler variational problem  \eqref{nustar2} holds for all $\beta > 0$. However, the original variational problem of \eqref{nustar} involves not only the $q$-th moment constraint $m_q(\nu)\in C$, but also an additional constraint on the $p$-th moment, $m_p(\nu)\le 1$. To simplify the variational problem \eqref{nustar}, it suffices to consider values of $\beta$ small enough such that $m_p(\mu_{q,\beta}) \le 1$, so that the maximizer of \eqref{nustar2} is also the maximizer of \eqref{nustar} when $C=[0,\beta]$. It is easy to see that
\begin{equation}\label{bpqcalc}
  m_p(\mu_{q,\beta_{p,q}}) = \beta_{p,q}^{p/q} q^{p/q}  \frac{\Gamma(\frac{p+1}{q})}{\Gamma(\frac{1}{q})} = 1.
\end{equation}
That is, $\beta_{p,q}$ is the threshold value such that for $\beta \le \beta_{p,q}$, we have $m_p(\mu_{q,\beta}) \le 1$. Note however, that this discussion of $\beta_{p,q}$ (and thus, Theorem \ref{cor-asymp}) is valid only for $p\in[1,\infty)$, since for $p=\infty$, we have $m_\infty(\mu_{q,\beta}) = \infty$ for all $q < \infty$ and $\beta > 0$. That is, for $q<\infty$, there is no value of $\beta_{\infty,q} > 0$ such that an analog of \eqref{bpqcalc} can hold for $p=\infty$.
\end{remark}

\begin{remark}[``large" $\beta$"]\label{rmk-beta2}
Now suppose $ \beta \ge \overline{\beta}_{p,q}  \doteq m_q(\mu_p)$. An easy consequence of Lemma \ref{lem-maxent} is that
\begin{equation}\label{nustar3}
\mu_p = \arg\max\{ h(\nu) : m_p(\nu) \le 1\}.
\end{equation}
For $\beta \ge \overline{\beta}_{p,q}$ and $C=[0,\beta]$, the $q$-th moment constraint of \eqref{nustar}  is automatically satisfied by $\nu=\mu_p$, so the maximizer of \eqref{nustar3} is also the maximizer of \eqref{nustar}. In this case, the ``conditional" limit of Proposition \ref{prop-clln} is equivalent to the ``unconditional" limit of  Proposition \ref{lem-pbm}. In other words, for sufficiently large $\beta$, the $\ell^q$ norm conditioning event of \eqref{condlaw} is extraneous.
\end{remark}

\begin{remark}[``intermediate" $\beta$]\label{rmk-beta3}
As for $C=[0,\beta]$ and $\beta_{p,q} < \beta < \overline{\beta}_{p,q}$, this regime is less amenable to an immediate geometric interpretation. Whereas the small $\beta$ regime of Theorem \ref{cor-asymp} and Remark \ref{rmk-beta} allows an $\ell^q$ sphere interpretation via the measure $\mu_{q,\beta}$, and the large $\beta$ regime of Remark \ref{rmk-beta2} allows an $\ell^p$ sphere interpretation via the measure $\mu_p$, we have different behavior in the intermediate $\beta$ regime. Recall that the maximum entropy considerations of Lemma \ref{lem-maxent} imply that the unique maximizer of Proposition \ref{prop-clln} is of the form 
\begin{equation}\label{nustarform}
  \nu_*(dx) = \exp\left(-1-\kappa_0 -\kappa_p|x|^p - \kappa_q|x|^q\right),
\end{equation}
with constants $\kappa_0,\kappa_p,\kappa_q$ such that $\nu_*$ is a probability measure that satisfies $m_p(\nu_*)\le 1$ and $m_q(\nu_*) \le \beta$. To gain insight on these constants, consider the following cases:
\begin{itemize}
\item if $\kappa_p = 0$ and $\kappa_q > 0$, then the complementary slackness condition of \eqref{compslack} implies $m_q(\nu_*) =\beta$, so the form \eqref{nustarform} implies $\nu_* = \mu_{q,\beta}$. But this is not a feasible solution, since $m_p(\mu_{q,\beta}) > 1$ for $\beta > \beta_{p,q}$;
\item similarly, if $\kappa_p > 0$ and $\kappa_q = 0$, then we have $\nu_* = \mu_p$, but this is not a feasible solution since $m_q(\mu_p) = \overline{\beta}_{p,q} > \beta$;
\item lastly, if $\kappa_p=0$ and $\kappa_q =0$, then $\nu_*$ is not a probability measure for any choice of $\kappa_0$.
\end{itemize}
Therefore, it must be the case that $\kappa_p > 0$ \emph{and} $\kappa_q > 0$, which implies that $\nu_*$ of \eqref{nustarform} is not of the form $\mu_{r,b}$ for any $r\in[1,\infty)$, $b> 0$. Instead, $\nu_*$ is of an exponential family that is genuinely different from the generalized normal distributions.
\end{remark}

In this paper, we have only discussed applications of Proposition \ref{prop-clln} to intervals $C$ of the form $C=[0,\beta]$, and primarily for small $\beta$, with some discussion of larger $\beta$ in Remark \ref{rmk-beta2} and Remark \ref{rmk-beta3}. We leave for future work the question of finding other examples of intervals $C\subset\R$ which lead to an explicit and meaningful maximizing measure $\nu_*$, and other geometric conditional laws.


\bibliography{ldpbib}
\bibliographystyle{amsplain}

\end{document}